\documentclass[a4paper, 12pt]{amsart}
\usepackage{amsmath}
\usepackage{amssymb}
\usepackage{bbm}
\usepackage{color,hyperref}

\theoremstyle{plain}
\newtheorem{theorem}{\bf Theorem}[section]

\newtheorem{lemma}[theorem]{\bf Lemma}

\theoremstyle{definition}

\newcommand{\N}{\mathbb N}
\newcommand{\Z}{\mathbb Z}

 \DeclareMathOperator{\ord}{ord}

 \DeclareMathOperator{\supp}{supp}

\renewcommand{\t}{\, | \,}

\numberwithin{equation}{section}

\begin{document}

\title[The structure of sequences]{The structure of sequences with zero-sum subsequences of the same length on finite abelian groups of rank two}

\author[W. Hui]{Wanzhen Hui}
\address{Sichuan Normal University, School of Mathematical Sciences, Chengdu, Sichuan, 610066, China} \email{huiwanzhen@163.com (W. Hui)}

\author[X. Li]{Xue Li*}
\address{Tianjin University of Commerce, College of Science, Tianjin, 300134, China} \email{lixue931006@163.com (X. Li)}

%\thanks{}



\thanks{*Corresponding author: Xue Li, Email: lixue931006@163.com}

\begin{abstract}
Let $G$ be an additive finite abelian group, and let $\mathrm{disc}(G)$
denote the smallest positive integer $t$ with the property that
every sequence $S$ over $G$ with length $|S|\geq t $ contains two nonempty
zero-sum subsequences of distinct lengths.
In recent years, Gao et al. established the exact value of $\mathrm{disc}(G)$
for all finite abelian groups of rank $2$ and resolved the corresponding
inverse problem for the group $C_n \oplus C_n$. In this paper,
we characterize the structure of sequences $S$ over $G = C_n \oplus C_{nm}$ (where $m\geq 2$)
when $|S| = \mathrm{disc}(G)- 1$ and all nonempty zero-sum subsequences of $S$ have the same length.
 \end{abstract}

\maketitle

\section{Introduction}
\bigskip

Throughout this paper, let $G$ be an additive finite abelian group.
We denote by $C_{n}$ the cyclic group of $n$ elements, and denote by
$C_{n}^{r}$ the direct sum of $r$ copies of $C_{n}$.
%By the
%Fundamental Theorem of Finite Abelian Groups, either $|G|=1$, or
%$G\cong C_{n_{1}}\oplus\cdots\oplus C_{n_{r}}$ with
%$1<n_{1}\mid\cdots\mid n_{r}$, where $r=\mathsf r(G)$ denotes the rank of $G$ and
%$n_{r}=\mathrm{exp}(G)$ is the exponent of $G$.
%We set
%$$\mathsf D^{*}(G)=1+\sum_{i=1}^{r}(n_{i}-1).$$

Let $p$ be a prime number. An old conjecture posed by Graham states that if $S$
is a sequence of length $|S|=p$ over $C_{p}$ such that all nonempty
zero-sum subsequences of $S$ have the same length, then $S$ takes at
most two distinct terms. In 1976, P. Erd\H{o}s and E. Szemer\'{e}di
\cite{Er-Sz76} showed that Graham's conjecture holds for sufficiently
large $p$. In 2010, W. Gao, Y. Hamidoune and G. Wang \cite{Ga-Ha-Wa10}
proved Graham's conjecture in full generality. Furthermore, they extended this
result to all positive integers.
Subsequently,  D. Grynkiewicz \cite{Gr11} provided an alternative proof.
In 2012, B. Girard \cite{Gir12} posed the problem of determining
the smallest integer $t$, which is denoted by $\mbox{disc}(G)$,
such that every sequence $S$ over $G$ of length $|S|\geq t$ has two nonempty zero-sum subsequences
of distinct lengths. Since then, $\mbox{disc}(G)$ has been systematically studied
by numerous authors,
and its exact value has been determined for several classes of finite abelian groups,
including the groups of rank at most two, the groups of very large exponent
compared to $|G|/\exp(G)$, elementary $2$-groups, additional special abelian
$p$-groups and certain groups of rank three (see \cite{GHLYZ, Ga-Li-Zh-Zh16, Ga-Zh-Zh15, L-Y24}).

On the other hand, Gao et al. \cite{Ga-Li-Zh-Zh16} considered the inverse problem
associated with $\mbox{disc}(G)$. In particular, they investigate
the set of all positive integers $t$, denoted by $\mathcal{L}_{1}(G)$,
such that there is a sequence $S$ over $G$ with length
$\mbox{disc}(G)-1$ and all nonempty zero-sum subsequences
of $S$ have the same length $t$.
They conjectured that $|\mathcal{L}_{1}(G)|=1$
for any finite abelian group.
In 2020, Gao et al. \cite{GHLYZ}
proved that $\mathcal{L}_{1}(G)=\{\mathrm{exp}(G)\}$ for the
abelian groups of rank at most two, $C_{mp^{n}}\oplus H$ with $m$ being a positive integer
and $H$ being a $p$-group with $\mathsf D(H)\leq p^{n}$, the finite abelian groups of
very large exponent compared to $|G|/\exp(G)$.
Moreover, they disproved this conjecture by demonstrating that
$|\mathcal{L}_{1}(G)|\geq 2$ for certain abelian $p$-groups.
To gain a deeper understanding of the sequence structures on these groups,
they have respectively characterized the structure of sequences of length
$\mbox{disc}(G)-1$ where all nonempty zero-sum subsequences have the same length
on the cyclic group $C_n$ and the group $C_{n}\oplus C_n$.
Recently, X. Li and Q. Yin \cite{L-Y24} have successfully
extended the scope of application of the conjecture to some groups of rank $3$,
including the group $C_2\oplus C_{2m}\oplus C_{2mn}$ and $C_3\oplus C_{6m}\oplus C_{6m}$,
where $m$ and $n$ are positive integers with $m|n$.
%However, as so far, the exact values of $\mbox{disc}(G)$ and the corresponding inverse problem
%are known for some special finite abelian groups as mention above,
%Especially, if $r(G)=3$, $\mbox{disc}(G)$ is only known for fairly special types of groups.
%and only for fairly special types of groups of rank three.
Currently, research on such inverse zero-sum problems remains insufficient,
and existing results are largely confined to specific group structures.
To overcome this constraint, it is necessary to employ novel methodologies to
characterize the structure of extremal sequences over more finite abelian groups
in which all nonempty zero-sum subsequences have the same length.

In this paper, we consider more general finite abelian groups $G$ of rank $2$,
and we mainly characterize the structure of the sequence $S$
when $|S| = \mathrm{disc}(G)- 1$ and all nonempty zero-sum subsequences of
$S$ have the same length.

Our main result is as follows.

\begin{theorem}\label{t1}
Let
$G=C_{n}\oplus C_{nm}$ with $n, m\geq2$ be integers. Let $S$ be a sequence
over $G$ with length $\mathrm{disc}(G)-1$ and all nonempty zero-sum
subsequences of $S$ have the same length.
Then there exists a generating set $\{g_1, g_2\}$ of $G$ with $\ord(g_2)=nm$ such that $S$ has one of the following forms.

\begin{itemize}
\item[(1)] $S=g_2^{2nm-1}\prod_{i=1}^{n-1}(x_ig_2+g_1)$, where $\ord(g_1)=n$ and $x_1,\ldots, x_{n-1}\in[0,nm-1]$.

\item[(2)] $S=g_1^{n-2}g_2^{2nm-1}(-(n-1)g_1+g_2)$.

\item[(3)] $S=g_1^{n-1}g_2^{2nm-1}$.

\item[(4)] $S=g_1^{2nm-1}\prod_{i=1}^{n-1}(-y_ig_1+g_2)$, where $\ord(g_1)=nm$, and $\Sigma_{i=1}^{n-1}y_i\in[0,n-1]$.

\item[(5)] $S=g_1^{sn+tn-1}g_2^{2nm+n(1-s)-tn-1}$, where $\ord(g_1)=nm$, $s\in[1,m]$ and $t\in[0,m]$.
\end{itemize}

\end{theorem}

The rest of the paper is organized as follows. Section 2 provides some basic notation and preliminaries. Section 3 gives the proof of our main result.

\section{Preliminaries}

Throughout this paper, our notation and terminology are consistent with
\cite{Ga-Ge06, Ge-Ha06} and we briefly present some key concepts.
Let $\mathbb{Z}$ denote the set of integers, and let $\mathbb{N}$ denote
the set of positive integers, $\mathbb{N}_{0} = \mathbb{N}\cup\{0\}$. 
For  real numbers $a, b \in \mathbb R$, we
set $[a, b] = \{ x \in \mathbb Z \colon a \le x \le b\}$.

 Let $G$ be an
abelian group. A family $(e_i)_{i \in I}$ of nonzero elements of
$G$ is said to be {\it independent} if
\[
\sum_{i \in I} m_ie_i =0 \quad \text{implies} \quad m_i e_i =0 \quad
\text{for all }  i \in I, \quad \mbox{ where } m_i\in \Z\,.
\]
If $I = [1, r]$ and $(e_1, \ldots, e_r)$ is independent, then we
simply say that $e_1, \ldots, e_r$ are independent elements of $G$.
The tuple $(e_i)_{i \in I}$ is called a {\it basis} if $(e_i)_{i \in
I}$ is independent and $\langle \{e_i \colon i \in I \} \rangle = G$.
If  $1 < |G| < \infty$, then we have
\[
G \cong C_{n_1} \oplus \cdots \oplus C_{n_r}  ,
\]
where $C_n$ denotes a cyclic group with $n$ elements, $i \in \N$ and $1<n_1\t \cdots\t n_r$. Then $r=\mathrm r(G)$ is the {\it rank}  of $G$ and
$n_r=\exp(G)$ is the \emph{exponent} of $G$.

%Let $G_0\subset G$ be a nonempty subset and let $\langle G_0
%\rangle \subset G$ be the subgroup generated by $G_0$.
%Let $
%S =g_1\cdot\ldots\cdot g_{\ell}
%$
%be a sequence over $G_0$. Let $ g$ be an element of $G$. We set $g+S=(g+g_1)\cdot\ldots\cdot (g+g_{\ell})$. Let $T$ be a subsequence of $S$. We denote
%$ST^{-1}=\prod_{g\in G_0}g^{\mathsf{v}_g(S)-\mathsf{v}_g(T)}$.
%If $1\leq|T|<|S|$, we say $T$ is a \emph{proper subsequence} of $S$. We
%call $S$
%\begin{itemize}
%\item a \emph{zero-sum free sequence} if there is no nonempty
%zero-sum subsequence of $S$;

%\item a \emph{minimal zero-sum sequence} if $|S|\ge 1$
%and  $S$ is a zero-sum sequence, but $S$ has no proper zero-sum subsequence.
%\end{itemize}
We denote by $\mathcal{F}(G)$ the free (abelian, multiplicative)
monoid with basis $G$. An element $S\in \mathcal{F}(G)$
is called a \emph{sequence} over $G$ and will be written in the form
$$S = g_{1}\cdot\ldots\cdot g_{l}=\prod_{g\in G}g^{\mathrm{v}_{g}(S)},$$
where $\mathrm{v}_{g}(S)\geq 0$ is called the \emph{multiplicity} of $g$ in $S$,
and we call
\begin{itemize}
\item $\mathrm{supp}(S)=\{g\in G\mid\mathrm{v}_{g}(S)>0\}$ the \emph{support} of $S$,
\item $|S| = l =\sum_{g\in G}\mathrm{v}_{g}(S)\in \mathbb{N}_{0}$ the \emph{length} of $S$,
%\item $\mathsf h(S)=\max \{\mathrm{v}_g(S)\mid g\in G \}$ the \emph{height} of $S$,
\item $\sigma(S)=\sum_{i=1}^{l}g_{i}=\sum_{g\in G}\mathrm{v}_{g}(S)g\in G$ the \emph{sum} of $S$,
\item $\Sigma_k(S)=\{\sum_{i\in I}g_i\mid I\subset[1, l] \ \text{with}\ |I|=k\}$ the \emph{set of $k$-term subsums} of $S$,
for all $k\in\mathbb N$,
\item $\Sigma_{\geq k}(S)=\bigcup_{j\geq k}\Sigma_j(S)$,
\item $\Sigma(S)=\Sigma_{\geq 1}(S)$ the \emph{set of all subsums} of $S$,
\item $T=\prod_{g\in G}g^{\mathrm{v}_{g}(T)}$ a \emph{subsequence} of $S$ if $\mathrm{v}_{g}(T)\leq\mathrm{v}_{g}(S)$ for all $g\in G$,
\item $T$ a \emph{proper subsequence} of $S$ if $T$ is a subsequence of $S$ and $1\leq|T|<|S|$,
\item $ST^{-1}=\prod_{g\in G}g^{\mathrm{v}_{g}(S)-\mathrm{v}_{g}(T)}$ the subsequence obtained from $S$ by deleting $T$,
\item $S$ a \emph{zero-sum sequence} if $\sigma(S)=0$,
\item $S$ a \emph{zero-sum free sequence} if there is no nonempty zero-sum subsequence of $S$,
\item $S$ a \emph{minimal zero-sum sequence} if it is zero-sum and has no proper zero-sum subsequence.
%\item $S$ a \emph{short zero-sum sequence} if $\sigma(S) = 0$ and $|S|\in[1, \mathrm{exp}(G)]$,
%\item two subsequences $T_1$ and $T_2$ of $S$  \emph{disjoint} if $T_1\mid ST_2^{-1}$.
\end{itemize}

For a finite abelian group $G$, let $\mathsf D(G)$ denote the Davenport constant of $G$,
which is defined as the smallest positive integer $d$ such that every
sequence over $G$ of length at least $d$ has a nonempty zero-sum
subsequence.

We next give several lemmas which will be used in the sequel.

\begin{lemma}\cite[Theorem 5.8.3]{Ge-Ha06}\label{eta}
Let $G=C_n\oplus C_{nm}$ with $n,m$ be integers. Then $\mathsf D(G)=n+nm-1$.
\end{lemma}

\begin{lemma}\cite[Theorem 1.2]{Ga-Zh-Zh15}\label{disc}
Let $G$ be a finite abelian group with $\mathrm r(G)\leq 2$. Then $\mathrm {disc}(G)=\mathsf D(G)+\exp(G)$.
\end{lemma}

\begin{lemma}\cite[Theorem 1.4]{GHLYZ}\label{l}
Let $G$ be a finite abelian group with $\mathrm r(G)\leq 2$. Then $\mathcal{L}_1(G)=\{\exp(G)\}$.
\end{lemma}

\begin{lemma}\cite[Proposition 5.1.4]{Ge-Ha06}\label{D(G)-1}
Let $G$ be a finite abelian group and let $S$ be a zero-sum free sequence over $G$ with $|S|=\mathsf D(G)-1$. Then $|\Sigma(S)|=|G|-1$.
\end{lemma}

\begin{lemma}\cite[Theorem 3.2]{S2010}\label{D(G)}
Let $G=C_{n_1}\oplus C_{n_2}$ be a finite abelian group with $1< n_1\mid n_2$ and $S$ be a sequence over $G$. Then $S$ is a minimal zero-sum sequence of length $|S|=n_1+n_2-1$ if and only if $S$ has one of the following forms:
\begin{itemize}
\item[(1)] $$S=e_j^{\ord(e_j)-1}\prod_{\nu=1}^{\ord(e_k)}(x_\nu e_j+e_k),$$
where  $\{e_1, e_2\}$ is a basis of $G$ with $\ord(e_i)=n_i$ for $i\in[1,2]$, $\{j,k\}=\{1,2\}$, $x_1,\ldots, x_{\ord(e_k)}\in [0, \ord(e_j)-1]$ and $x_1+\cdots+x_{\ord(e_k)}\equiv 1 \pmod {\ord(e_j)}$.
\item[(2)] $$S=g_1^{sn_1-1}\prod_{\nu=1}^{n_2+(1-s)n_1}(-x_\nu g_1+g_2),$$
where  $\{g_1, g_2\}$ is a generating set of $G$ with $\ord(g_2)=n_2$, $x_1,\ldots, x_{n_2+(1-s)n_1}\in [0, n_1-1]$ and $x_1+\cdots+x_{n_2+(1-s)n_1}=n_1-1$, $s\in [1, n_2/n_1]$ and either $s=1$ or $n_1g_1=n_1g_2$.
\end{itemize}
\end{lemma}

\medskip

\bigskip
\section{Proof of Theorem \ref{t1}}

In this section, we present the proof of our main result. To begin with, we establish a crucial lemma.

\begin{lemma}\label{empty}
Let $S$ be a sequence over a finite abelian group $G$ of length $|S|=\mathrm{disc}(G)-1$,
where all nonempty zero-sum subsequences of $S$ have the same length.
Suppose $T$ is a nonempty zero-sum subsequence of $S$. Then
$$\supp(T)\cap \Sigma_{\geq 2}(ST^{-1})=\emptyset.$$
\end{lemma}

\begin{proof}
Assume to the contrary that there exists a subsequence $T'\mid ST^{-1}$ with $|T'|\geq 2$ such that $\sigma(T')=g$, where $g\mid T$. Then $T'Tg^{-1}$ is a zero-sum subsequence of $S$ with length $|T'Tg^{-1}|> |T|$, which is a contradiction. Hence, $\supp(T)\cap \Sigma_{\geq 2}(ST^{-1})=\emptyset$.
\end{proof}

We are now in position to provide the proof for our main result.
\vskip .5 cm
{\bf Proof of Theorem \ref{t1}.}
By Lemmas \ref{eta} and \ref{disc}, we have $|S|=\mathrm{disc}(G)-1=\mathsf D(G)+\exp(G)-1=n+2nm-2$. And it follows from Lemma \ref{l} that all nonempty zero-sum subsequences of $S$ have the same length $nm$.

Since $|S|=n+2nm-2> \mathsf D(G)=n+nm-1$, there exists a zero-sum subsequence $T$ of $S$ with length $nm$ and $0\nmid S$. Then $|ST^{-1}|=n+nm-2=\mathsf D(G)-1$ and $ST^{-1}$ is zero-sum free. It follows from Lemma \ref{D(G)-1} that $$\Sigma(ST^{-1})=G\setminus\{0\}.$$

It is easy to see that $ST^{-1}(-\sigma(ST^{-1}))$ is a minimal zero-sum sequence of length $nm+n-1=\mathsf D(G)$. And by Lemma \ref{D(G)} we obtain
that $ST^{-1}(-\sigma(ST^{-1}))$ has one of the following forms:
\begin{equation}\label{1}
ST^{-1}(-\sigma(ST^{-1}))=e_u^{\ord(e_u)-1}\prod_{i=1}^{\ord(e_v)}(x_i e_u+e_v),
\end{equation}
where  $\{e_1, e_2\}$ is a basis of $G$ with $\ord(e_1)=n$ and $\ord(e_2)=nm$, $\{u,v\}=\{1,2\}$, $x_1,\ldots, x_{\ord(e_v)}\in [0, \ord(e_u)-1]$ and $x_1+\cdots+x_{\ord(e_v)}\equiv 1 \pmod {\ord(e_u)}$.

\begin{equation}\label{2}
ST^{-1}(-\sigma(ST^{-1}))=g_1^{sn-1}\prod_{i=1}^{nm+(1-s)n}(-y_i g_1+g_2),
\end{equation}
where  $\{g_1, g_2\}$ is a generating set of $G$ with $\ord(g_2)=nm$, $y_1,\ldots, y_{nm+(1-s)n}\in [0, n-1]$ and $y_1+\cdots+y_{nm+(1-s)n}=n-1$, $s\in [1, m]$ and either $s=1$ or $ng_1=ng_2$.

We now divide the remaining proof into the following four cases.

{\bf Case 1.} $ST^{-1}(-\sigma(ST^{-1}))$ is of the form (\ref{1}) with $u=1$ and $v=2$. It follows that $$ST^{-1}(-\sigma(ST^{-1}))=e_1^{n-1}\prod_{i=1}^{nm}(x_ie_1+e_2),$$
where $x_1,\ldots, x_{nm}\in [0, n-1]$ and $x_1+\cdots+x_{nm}\equiv 1 \pmod n$.

{\bf Subcase 1.1.} $ST^{-1}=e_1^{n-1}\prod_{i=1}^{nm-1}(x_ie_1+e_2)$. If $x_i\neq x_j$ for some $i\neq j\in [1,nm-1]$, then
$$\Sigma_{\geq 2}(ST^{-1})=G\setminus\{0,e_1\}.$$
By Lemma \ref{empty} and $0\nmid S$, we obtain that $T$ is of the form $e_1^{nm}$. Therefore, $e_1^{n}$ is a zero-sum subsequence of $S$ with length $n<nm$, which is a contradiction.

Next we assume that $x_1=\cdots =x_{nm-1}$, i.e. $ST^{-1}=e_1^{n-1}(x_1e_1+e_2)^{nm-1}$. Replacing $x_1e_1+e_2$ with $e_2$, we have that $ST^{-1}=e_1^{n-1}e_2^{nm-1}$ and
$$\Sigma_{\geq 2}(ST^{-1})=G\setminus\{0,e_1,e_2\}.$$
If $e_1\mid T$, then $e_1^{n}$ is a zero-sum subsequence of $S$ of length $n<nm$, a contradiction. By Lemma \ref{empty} and $0\nmid S$, we obtain that $T$ is of the form $e_2^{nm}$. Therefore,
$$S=e_1^{n-1}e_2^{2nm-1}.$$
Replacing $e_1$ with $g_1$ and $e_2$ with $g_2$, $\{g_1,g_2\}$ is a generating set of $G$ where $\ord(g_1)=n$ and $\ord(g_2)=nm$. Thus $S$ is of the form (1).

{\bf Subcase 1.2.} $ST^{-1}=e_1^{n-2}\prod_{i=1}^{nm}(x_ie_1+e_2)$ with $\Sigma_{i=1}^{nm}x_i\equiv 1\pmod n$. It follows that $x_i\neq x_j$ for some $i\neq j\in [1,nm]$.
Without loss of generality, we may assume that $0\leq x_1\leq\cdots\leq x_{nm}\leq n-1$.
If $x_i, x_j, x_k$ are pairwise distinct for some $i,j,k\in [1,nm]$ or $x_i-x_j\in[2,n-2]$ for some $i, j\in[1,nm]$, then
$$\Sigma_{\geq 2}(ST^{-1})=G\setminus\{0\}.$$
A contradiction with Lemma \ref{empty} and $0\nmid S$.

Next we assume that $x_i-x_j\in \{-1,0,1\}$ for every $i,j\in [1,nm]$. We may assume that $x_1=\cdots=x_l=x_{l+1}-1=\cdots=x_{nm}-1$, $l\in[1,nm-1]$. Thus $\Sigma_{i=1}^{nm}x_i=nmx_1+(nm-l)\equiv 1\pmod n$, it deduces that $l=nm-tn-1$, $t\in [0,m-1]$. So $ST^{-1}=e_1^{n-2}(x_1e_1+e_2)^{nm-tn-1}((x_1+1)e_1+e_2)^{tn+1}$.
Replacing $x_1e_1+e_2$ with $e_2$, we have that $ST^{-1}=e_1^{n-2}e_2^{nm-tn-1}(e_1+e_2)^{tn+1}$. Then
$$\Sigma_{\geq 2}(ST^{-1})=G\setminus\{0, e_2\}.$$
By Lemma \ref{empty} and $0\nmid S$, we obtain that $T$ is of the form $e_2^{nm}$. If $t\geq 1$ and $n\geq 3$, then $e_1^{n-2}(e_1+e_2)^2e_2^{nm-2}$ is a zero-sum subsequence of $S$ with length $nm+n-2> nm$, which is a contradiction.
Therefore $t=0$ or $n=2$. If $t=0$, then
$$S=e_1^{n-2}e_2^{2nm-1}(e_1+e_2).$$
Replacing $e_1$ with $g_1$ and $e_2$ with $g_2$, $\{g_1,g_2\}$ is a generating set of $G$ where $\ord(g_1)=n$ and $\ord(g_2)=nm$. Thus $S$ is of the form (1).

If $n=2$, then
$$S=e_2^{4m-2t-1}(e_1+e_2)^{2t+1}.$$
Replacing $e_1+e_2$ with $g_1$ and $e_2$ with $g_2$, $\{g_1,g_2\}$ is a generating set of $G$ where $\ord(g_1)=\ord(g_2)=nm$. Thus $S$ is of the form (5).

{\bf Case 2.} $ST^{-1}(-\sigma(ST^{-1}))$ is of the form (\ref{1}) with $u=2$ and $v=1$. It follows that $$ST^{-1}(-\sigma(ST^{-1}))=e_2^{nm-1}\prod_{i=1}^{n}(x_ie_2+e_1),$$
where $x_1,\ldots, x_{n}\in [0, nm-1]$ and $x_1+\cdots+x_{n}\equiv 1 \pmod {nm}$.
If $m=1$, then it reduces to Case 1. Therefore we assume that $m\geq 2$.

{\bf Subcase 2.1.} $ST^{-1}=e_2^{nm-1}\prod_{i=1}^{n-1}(x_ie_2+e_1)$. If $x_i\neq x_j$ for some $i\neq j\in [1,n-1]$, then
$$\Sigma_{\geq 2}(ST^{-1})=G\setminus\{0,e_2\}.$$
By Lemma \ref{empty} and $0\nmid S$, we conclude that $T$ is of the form $e_2^{nm}$. Therefore
$$S=e_2^{2nm-1}\prod_{i=1}^{n-1}(x_ie_2+e_1).$$
Replacing $e_1$ with $g_1$ and $e_2$ with $g_2$, $\{g_1,g_2\}$ is a generating set of $G$ where $\ord(g_1)=n$ and $\ord(g_2)=nm$. Thus $S$ is of the form (1).

Next we assume that $x_1=\cdots=x_{n-1}$, i.e. $ST^{-1}=e_2^{nm-1}(x_1e_2+e_1)^{n-1}$. If $x_1\pmod m=0$, it is easy to see that $\ord(x_1e_2+e_1)=n$. By replacing $x_1e_2+e_1$ with $e_1$, we have $ST^{-1}=e_2^{nm-1}e_1^{n-1}$ and it reduces to Case 1. Next we suppose $x_1\pmod m\in [1,m-1]$. Then
$$\Sigma_{\geq 2}(ST^{-1})=G\setminus\{0,e_2, x_1e_2+e_1\}.$$

If $x_1\pmod m\in [2,m-1]$ and $(x_1e_2+e_1)\mid T$, then $(x_1e_2+e_1)^ne_2^{nm-(nx_1\pmod{nm})}$ is a zero-sum subsequence of $S$ with length $nm+n-(nx_1\pmod{nm})<nm$, which is a contradiction. By Lemma \ref{empty} and $0\nmid S$, we obtain that $T$ is of the form $e_2^{nm}$. Therefore
$$S=e_2^{2nm-1}(x_1e_2+e_1)^{n-1}.$$
Replacing $e_1$ with $g_1$ and $e_2$ with $g_2$, $\{g_1,g_2\}$ is a generating set of $G$ where $\ord(g_1)=n$ and $\ord(g_2)=nm$. Thus $S$ is of the form (1).

Next we suppose that $x_1\pmod m=1$. Replacing $x_1e_1+e_2$ with $e_1+e_2$, we have $ST^{-1}=e_2^{nm-1}(e_1+e_2)^{n-1}$. By Lemma \ref{empty} and $0\nmid S$, we obtain that $T$ is of the form $e_2^{nm-tn}(e_1+e_2)^{tn}$ for $t\in[0,m]$. Therefore
$$S=e_2^{2nm-tn-1}(e_1+e_2)^{n+tn-1}.$$
Replacing $e_1+e_2$ with $g_1$ and $e_2$ with $g_2$, $\{g_1,g_2\}$ forms a generating set of $G$ where $\ord(g_1)=\ord(g_2)=nm$. Thus $S$ is of the form (5).

{\bf Subcase 2.2.} $ST^{-1}=e_2^{nm-2}\prod_{i=1}^{n}(x_ie_2+e_1)$ with $\Sigma_{i=1}^nx_i\equiv 1\pmod {nm}$. It follows that $x_i\neq x_j$ for some $i\neq j\in [1,n]$. Without loss of generality, we may assume that $0\leq x_1\leq \cdots\leq x_n\leq nm-1$. If $x_i, x_j, x_k$ are pairwise distinct for some $i,j,k\in[1,n]$, or if $x_i-x_j\in[2,nm-2]$ for some $i,j\in[1,n]$, we infer that
$$\Sigma_{\geq 2}(ST^{-1})=G\setminus\{0\}.$$
A contradiction with Lemma \ref{empty} and $0\nmid S$.

Next we assume that $x_i-x_j\in \{-1, 0,1\}$ for every $i,j\in[1,n]$. We may assume that $x_1=\cdots=x_l=x_{l+1}-1=\cdots=x_n-1$, $l\in[1,n-1]$. Thus $\Sigma_{i=1}^nx_i=nx_1+n-l\equiv 1\pmod {nm}$, it deduces that $l=n-1$ and $m\mid x_1$. So $ST^{-1}=e_2^{nm-2}(x_1e_2+e_1)^{n-1}((x_1+1)e_2+e_1)$.  By replacing $x_1e_2+e_1$ with $e_1$, we have $ST^{-1}=e_2^{nm-2}e_1^{n-1}(e_2+e_1)$. Thus it reduces to Case 1 and we are done.

{\bf Case 3.} $ST^{-1}(-\sigma(ST^{-1}))$ is of the form (\ref{2}) with $ng_1\neq ng_2$. By (\ref{2}), we have $s=1$. So we can write $$ST^{-1}(-\sigma(ST^{-1}))=g_1^{n-1}\prod_{i=1}^{nm}(-y_ig_1+g_2).$$
Note that $y_1, \ldots,y_{nm}\in[0,n-1]$ and $y_1+\cdots+y_{nm}=n-1$.

Let $\varphi: G\longrightarrow G/\langle g_2 \rangle$ denote the canonical epimorphism. Since $\ord(g_2)=nm$ and $\{g_1,g_2\}$ is a generating set of $G$, we have $n=\ord(\varphi(g_1))\mid \ord(g_1)$. If $\ord(g_1)=n$, then it reduces to Case 1. So we may assume that $n<\ord(g_1)\leq nm$. Suppose $g_1=xe_1+t_0g_2$ for some integer $t_0\in [0,nm-1]$ and $x\in[0, n-1]$, where $\{e_1, g_2\}$ is a basis of $G$. Then $ng_1=t_0ng_2$. Since $\ord(g_2)=nm$, we infer that there exists $t'\in[0,m-1]$ such that
$$ng_1=t_0ng_2=t'ng_2.$$
Since $ng_1\neq ng_2$, we obtain that $t'\in[2,m-1]$, it deduces that $m\geq 3$.

{\bf Subcase 3.1.} $ST^{-1}=g_1^{n-2}\prod_{i=1}^{nm}(-y_ig_1+g_2)$.
Without loss of generality, we assume that $n-1\geq y_1\geq\cdots \geq y_j> y_{j+1}=\cdots=y_{nm}=0$. Since $y_1+\cdots+y_{nm}=n-1$, we have $j\in [1,n-1]$.

If $j\in[2,n-2]$, we obtain that
$$\Sigma_{\geq 2}(ST^{-1})=G\setminus\{0\}.$$
A contradiction with Lemma \ref{empty} and $0\nmid S$.

If $j=1$, i.e. $ST^{-1}=g_1^{n-2}g_2^{nm-1}(-(n-1)g_1+g_2)$, then we have
$$\Sigma_{\geq 2}(ST^{-1})=G\setminus\{0, g_2\}.$$
By Lemma \ref{empty} and $0\nmid S$, we obtain that $T$ is of the form $g_2^{nm}$. Therefore
$$S=g_1^{n-2}g_2^{2nm-1}(-(n-1)g_1+g_2).$$
Thus $S$ is of the form (2).

If $j=n-1$, i.e. $ST^{-1}=g_1^{n-2}g_2^{nm-n+1}(-g_1+g_2)^{n-1}$, then we have
$$\Sigma_{\geq 2}(ST^{-1})=G\setminus\{0, -g_1+g_2\}.$$
By Lemma \ref{empty} and $0\nmid S$, we conclude that $T$ is of the form $(-g_1+g_2)^{nm}$.
Furthermore, $g_2^{(t'-1)n}(-g_1+g_2)^{n}$ is a zero-sum subsequence of $S$ with length $t'n<nm$,
which is a contradiction.

{\bf Subcase 3.2.} $ST^{-1}=g_1^{n-1}\prod_{i=1}^{nm-1}(-y_i g_1+g_2)$. Without loss of generality, we assume that
$n-1\geq y_1\geq\cdots \geq y_j> y_{j+1}=\cdots=y_{nm-1}=0$. Since $y_1+\cdots+y_{nm}=n-1$, we have $j\in [0,n-1]$.

If $j\in[1,n-1]$, we may first assume that $\Sigma_{i=1}^jy_i=n-1$, so we have
$$\Sigma_{\geq 2}(ST^{-1})=G\setminus\{0\}.$$
A contradiction with Lemma \ref{empty} and $0\nmid S$.

Next we consider that $\Sigma_{i=1}^jy_i<n-1$, then we have
$$\Sigma_{\geq 2}(ST^{-1})=G\setminus\{0,g_1\}.$$
Again by Lemma \ref{empty} and $0\nmid S$, we conclude that $T$ is of the form $g_1^{nm}$. Furthermore, $g_1^{nm-n+\sum_{i=1}^jy_i}\prod_{i=1}^{t'n}(-y_i g_1+g_2)$ is a zero-sum subsequence of $S$ with length $nm+t'n-n+\sum_{i=1}^jy_i>nm$, which is a contradiction.

If $j=0$, i.e. $ST^{-1}=g_1^{n-1}g_2^{nm-1}$, it follows that
$$\Sigma_{\geq 2}(ST^{-1})=G\setminus\{0, g_1, g_2\}.$$
If $g_1\mid T$, then $g_1^ng_2^{nm-t'n}$ is a zero-sum subsequence of $S$ with length $nm-t'n+n<nm$, which is a contradiction. By Lemma \ref{empty} and $0\nmid S$, we obtain that $T$ is of the form $g_2^{nm}$.  Therefore
$$S=g_1^{n-1}g_2^{2nm-1}.$$
Thus $S$ is of the form (3).

{\bf Case 4.} $ST^{-1}(-\sigma(ST^{-1}))$ is of the form (\ref{2}) with $ng_1= ng_2$. It follows that $$ST^{-1}(-\sigma(ST^{-1}))=g_1^{sn-1}\prod_{i=1}^{nm+(1-s)n}(-y_i g_1+g_2),$$
where $y_1, \ldots, y_{nm+(1-s)n}\in [0, n-1]$ and $y_1+\cdots+y_{nm+(1-s)n}=n-1$.

Similar to the proof of Case 3, we have $n\mid \ord(g_1)$. Since $ng_1=ng_2$ and $\ord(g_2)=nm$, we have $\frac{\ord(g_1)}{n}=\ord(ng_1)=\ord(ng_2)=m$, and so it deduces that $\ord(g_1)=nm$. It is easy to see that $\ord(-g_1+g_2)=n$.

{\bf Subcase 4.1.} $ST^{-1}=g_1^{sn-2}\prod_{i=1}^{nm+(1-s)n}(-y_i g_1+g_2)$. Without loss of generality, we assume that $n-1\geq y_1\geq\cdots \geq y_j> y_{j+1}=\cdots=y_{nm+(1-s)n}=0$. Since $y_1+\cdots+y_{nm+(1-s)n}=n-1$, we have $j\in [1,n-1]$.

 If $j\in[2,n-2]$, we obtain that
$$\Sigma_{\geq 2}(ST^{-1})=G\setminus\{0\}.$$
A contradiction with Lemma \ref{empty} and $0\nmid S$.

Suppose $j=1$, i.e. $ST^{-1}=g_1^{sn-2}g_2^{nm+n(1-s)-1}(-(n-1)g_1+g_2)$. If $s\geq 2$, we have
$$\Sigma_{\geq 2}(ST^{-1})=G\setminus\{0\}.$$
A contradiction with Lemma \ref{empty} and $0\nmid S$.

If $s=1$, i.e. $ST^{-1}=g_1^{n-2}g_2^{nm-1}(-(n-1)g_1+g_2)$.
It follows that
$$\Sigma_{\geq 2}(ST^{-1})=G\setminus\{0, g_2\}.$$
By Lemma \ref{empty} and $0\nmid S$, we obtain that $T$ is of the form $g_2^{nm}$. Therefore
$$S=g_1^{n-2}g_2^{2nm-1}(-(n-1)g_1+g_2).$$
Thus $S$ is of the form (2).

If $j=n-1$, i.e. $ST^{-1}=g_1^{sn-2}g_2^{nm-sn+1}(-g_1+g_2)^{n-1}$, then by replacing $g_1$ with $e_2$ and $-g_1+g_2$ with $e_1$, we see that $\{e_1,e_2\}$ is a basis of $G$. It follows that $ST^{-1}=e_1^{n-1}e_2^{nm-sn+1}(e_1+e_2)^{sn-2}$, and this reduces to Case 1.

{\bf Subcase 4.2.} $ST^{-1}=g_1^{sn-1}\prod_{i=1}^{nm+(1-s)n-1}(-y_i g_1+g_2)$. Without loss of generality, we assume that
$n-1\geq y_1\geq\cdots \geq y_j> y_{j+1}=\cdots=y_{nm+(1-s)n-1}=0$. Since $y_1+\cdots+y_{nm+(1-s)n}=n-1$, we have $j\in [0,n-1]$.

If $j\in[1,n-2]$, we first assume that $\Sigma_{i=1}^jy_i=n-1$ and $s\leq m-1$, it then follows that
$$\Sigma_{\geq 2}(ST^{-1})=G\setminus\{0\}.$$
A contradiction with Lemma \ref{empty} and $0\nmid S$.

Next, considering the cases where $\Sigma_{i=1}^jy_i<n-1$ or $s=m$, we infer that
$$\Sigma_{\geq 2}(ST^{-1})=G\setminus\{0,g_1\}.$$
By Lemma \ref{empty} and $0\nmid S$, we conclude that $T$ is of the form $g_1^{nm}$. If $s<m$, then $g_1^{nm-n+\sum_{i=1}^jy_i}\prod_{i=1}^{n}(-y_i g_1+g_2)$ is a zero-sum subsequence of $S$ with length $nm+\sum_{i=1}^jy_i>nm$,
which is a contradiction. Therefore $s=m$ and
$$S=g_1^{2nm-1}\prod_{i=1}^{n-1}(-y_i g_1+g_2).$$
Thus $S$ is of the form (4).

If $j=0$, i.e. $ST^{-1}=g_1^{sn-1}g_2^{nm+n(1-s)-1}$, it follows that
$$\Sigma_{\geq 2}(ST^{-1})=G\setminus\{0, g_1, g_2\}.$$
By Lemma \ref{empty} and $0\nmid S$, we conclude that $T$ is of the form $g_1^{tn}g_2^{nm-tn}$ for $t\in[0,m]$. Therefore
$$S=g_1^{sn+tn-1}g_2^{2nm+n(1-s)-tn-1}.$$
Thus $S$ is of the form (5).

Suppose $j=n-1$, i.e. $ST^{-1}=g_1^{sn-1}(-g_1+g_2)^{n-1}g_2^{nm-sn}$. Replacing $g_1$ with $e_2$ and $-g_1+g_2$ with $e_1$, we see that  $\{e_1,e_2\}$ is a basis of $G$. It follows that $ST^{-1}=e_1^{n-1}e_2^{nm-sn}(e_1+e_2)^{sn-1}$, which reduces to Case 1.
\qed

\bigskip
\noindent {\bf Acknowledgements.}
This work was supported by the Sichuan Science and Technology Program (Grant No. 2024NSFSC2051).

%\bibliography{ger,hk,fact,zerosum,ideal}
%\bibliographystyle{amsplain}

\providecommand{\bysame}{\leavevmode\hbox to3em{\hrulefill}\thinspace}
\providecommand{\MR}{\relax\ifhmode\unskip\space\fi MR }
\providecommand{\MRhref}[2]{%
  \href{http://www.ams.org/mathscinet-getitem?mr=#1}{#2}
}
\providecommand{\href}[2]{#2}

\end{document}